\documentclass[10pt]{amsart}
\usepackage{graphicx,amssymb,amsmath,amsfonts,amscd,hyperref}
\newtheorem{thm}{Theorem}[section]
\newtheorem{cor}[thm]{Corollary}
\newtheorem{lem}[thm]{Lemma}
\newtheorem{prop}[thm]{Proposition}

\newtheorem{rem}[thm]{Remark}

\newcommand{\Sym}{\mathrm{Sym}}

\newcommand{\Fq}{\mathbb F_q}
\newcommand{\Fqr}{{\mathbb F}_{q^r}}
\newcommand{\CC}{\mathbb C}
\newcommand{\FFq}{{\bar{\mathbb F}_q}}
\newcommand{\QQ}{\bar{\mathbb Q}_\ell}
\newcommand{\R}{\mathrm R}

\newcommand{\Gal}{\mathrm {Gal}}
\newcommand{\AAA}{\mathbb A}
\newcommand{\PP}{\mathbb P}
\newcommand{\FF}{\mathcal F}
\newcommand{\GG}{\mathbb G}
\newcommand{\GGG}{\mathcal G}

\newcommand{\HH}{\mathrm H}
\newcommand{\HHH}{{\mathcal H}}
\newcommand{\LL}{{\mathcal L}}

\newcommand{\Tr}{\mathrm {Tr}}
\newcommand{\Nm}{\mathrm N}

\newcommand{\Frob}{\mathrm{Frob}}
\newcommand{\Swan}{\mathrm{Swan}}

\begin{document}
\title{Estimates for exponential sums with a large automorphism group}
\author{Antonio Rojas-Le\'on}
\address{Departamanto de \'Algebra,
Universidad de Sevilla, Apdo 1160, 41080 Sevilla, Spain}
\address{E-mail: arojas@us.es}
\thanks{Partially supported by P08-FQM-03894 (Junta de Andaluc\'{\i}a), MTM2007-66929 and FEDER}

\begin{abstract}
 We prove some improvements of the classical Weil bound for one variable additive and multiplicative character sums associated to a polynomial over a finite field $k=\Fq$ for two classes of polynomials which are invariant under a large abelian group of automorphisms of the affine line $\AAA^1_k$: those invariant under translation by elements of $k$ and those invariant under homotheties with ratios in a large subgroup of the multiplicative group of $k$. In both cases, we are able to improve the bound by a factor of $\sqrt{q}$ over an extension of $k$ of cardinality sufficiently large compared to the degree of $f$.
\end{abstract}

\maketitle

\section{Introduction}

Let $k=\Fq$ be a finite field with $q$ elements. As a consequence Weil's bound for the number of rational points on a curve over $k$, one can obtain estimates for character sums defined on the affine line $\AAA^1_k$ (cf. \cite{hasse1935theorie},\cite{weil1948some}). Let us describe the precise results.

Let $f\in k[x]$ be a polynomial of degree $d$ and $\psi:k\to\CC^\star$ a non-trivial additive character. Consider the sum $\sum_{x\in k}\psi(f(x))$ (and, more generally, $\sum_{x\in k_r}\psi(\Tr_{k_r/k}(f(x)))$ for a finite extension $k_r$ of $k$ of degree $r$). Then, if $d$ is prime to $p$, we have the estimate
$$
\left|\sum_{x\in k_r}\psi(\Tr_{k_r/k}(f(x)))\right|\leq (d-1)q^{\frac{r}{2}}.
$$

If $d$ is divisible by $p$, we can reduce to the previous case using the following trick. Since $t\mapsto \psi(t^p)$ is a non-trivial additive character, there must be some $a\in k$ such that $\psi(t^p)=\psi(at)$ for every $t\in k$. If $f(x)=a_d x^d+a_{d-1}x^{d-1}+\cdots$ with $d=ep$, let $b_d\in k$ be such that $b_d^p=a_d$, then
$$
\psi(\Tr_{k_r/k}(f(x)))=\psi(\Tr_{k_r/k}((b_dx^e)^p))\psi(\Tr_{k_r/k}(f(x)-a_dx^d))=
$$
$$
=\psi(\Tr_{k_r/k}(b_dx^e)^p)\psi(\Tr_{k_r/k}(f(x)-a_dx^d))=\psi(a\cdot\Tr_{k_r/k}(b_dx^e))\psi(\Tr_{k_r/k}(f(x)-a_dx^d))=
$$
$$
=\psi(\Tr_{k_r/k}(f(x)-a_dx^d+ab_dx^e)).
$$
We keep reducing the polynomial in this way until we get a polynomial with degree $d'$ prime to $p$. Then we apply the prime to $p$ case and obtain an estimate
$$
\left|\sum_{x\in k_r}\psi(\Tr_{k_r/k}(f(x)))\right|\leq (d'-1)q^{\frac{r}{2}}.
$$
except when $d'$ is zero (that is, when $f=c+g^p-ag$ for some constant $c$ and sone $g\in k[x]$). If the character $\psi$ is obtained from a character of the prime subfield ${\mathbb F}_p$ by pulling back via the trace map, then $a=1$.

Similarly, if $\chi:k^\star\to\CC^\star$ is a multiplicative character of order $m>1$ and $f\in k[x]$ is not an $m$-th power, we have an estimate
$$
\left|\sum_{x\in k_r}\chi(\Nm_{k_r/k}(f(x)))\right|\leq (e-1)q^{\frac{r}{2}}\leq (d-1)q^{\frac{r}{2}}
$$
where $e$ is the number of distinct roots of $f$.

In this article we will improve these estimates for a special class of polynomials: those which are either translation invariant or homothety invariant, that is, either $f(x+\lambda)=f(x)$ for every $\lambda\in k$ or $f(\lambda x)=f(x)$ for every $\lambda\in k^\star$ (or every $\lambda$ in a large subgroup of $k^\star$). For such polynomials, there is a large abelian group $G$ of automorphisms of $\AAA^1_k$ such that $f\circ\sigma=f$ for every $\sigma\in G$. 

On the level of $\ell$-adic cohomology, this gives an action of $G$ on the pull-back by $f$ of the Artin-Schreier and Kummer sheaves associated to $\psi$ and $\chi$ respectively \cite[1.7]{deligne569application}, so they induce an action on their cohomology. The character sums can be expressed as the trace of the geometric $k_r$-Frobenius action on this cohomology, by Grothendieck's trace formula. The above estimates are a consequence of the fact that this action has all eigenvalues of archimedean absolute value $\leq q^{\frac{r}{2}}$. Precisely, if $S_r=\sum_{x\in k_r}\psi(\Tr_{k_r/k}(f(x)))$ (respectively $U_r=\sum_{x\in k_r}\chi(\Nm_{k_r/k}(f(x)))$) the $L$-functions
$$
L(\psi,f;T):=\exp\sum_{r\geq 1} S_r\frac{T^r}{r}
$$
and
$$
L(\chi,f;T):=\exp\sum_{r\geq 1} U_r\frac{T^r}{r}
$$
are the polynomials $\det(1-T\cdot\Frob_k|\HH^1_c(\AAA^1_{\bar k},f^\star\LL_\psi))$ and $\det(1-T\cdot\Frob_k|\HH^1_c(\AAA^1_{\bar k},f^\star\LL_\chi))$, of degree $d'-1$ and $e-1$ respectively.

Now under the action of the abelian group $G$, this cohomology splits as a direct sum of eigenspaces for the different characters of $G$. Under certain generic conditions, it is natural to expect some cancellation among the traces of the Frobenius actions on these eigenspaces, thus giving a substantial improvement of Weil's estimate if $G$ is large (namely by a $\sqrt{\# G}$ factor). Compare \cite{rojas-leon-biotwbfac}, where an improvement for the Weil estimate for the number of rational points on Artin-Schreier curves was obtained using the same arguments we apply in this article.

For the translation invariant case (sections 2 ans 3), we obtain this improvement using the local theory of $\ell$-adic Fourier transform \cite{laumon1987transformation} and Katz' computation of the geometric monodromy groups for some families of exponential sums \cite{katz-monodromy}, \cite{katz1990esa}. The argument is similar to that in \cite{rojas-leon-biotwbfac}. For the homothety invariant case (sections 4 and 5), we use Weil descent together with certain properties of the convolution of sheaves on $\GG_{m,k}$.  
 
Throughout this article, $k=\Fq$ will be a finite field of characteristic $p$, $\bar k=\FFq$ a fixed algebraic closure and $k_r=\Fqr$ the unique extension of $k$ of degree $r$ in $\bar k$. We will fix a prime $\ell\neq p$, and work with $\ell$-adic cohomology. In order to speak about weights without ambiguity, we will fix a field isomorphism $\iota:\QQ\to\CC$. We will use this isomorphism to identify $\QQ$ and $\CC$ without making any further mention to it. When we speak about weights, we will mean weights with respect to the chosen isomorphism $\iota$.

\section{Additive character sums for translation invariant polynomials}

Let $f\in k[x]$ be a polynomial. $f$ is said to be \emph{translation invariant} if $f(x+a)=f(x)$ for every $a\in k$.

\begin{lem}\label{translationinvariant}
 Let $f\in k[x]$. The following conditions are equivalent:
\begin{itemize}
 \item[(a)] $f$ is translation invariant.
 \item[(b)] There exists $g\in k[x]$ such that $f(x)=g(x^q-x)$.
\end{itemize}
\end{lem}

\begin{proof}
 $(b)\Rightarrow (a)$ is clear. Suppose that $f$ is translation invariant. If the degree of $f$ is $<q$, the polynomial $f(x)-f(0)$ has at least $q$ roots (all elements of $k$) and degree $<q$, so it is identically zero. So $f$ is the constant polynomial $f(0)$. Otherwise, we can write $f(x)=(x^q-x)h(x)+r(x)$ with $\deg(r)<q$. For every $a\in k$ we have then $f(x+a)=(x^q-x)h(x+a)+r(x+a)=(x^q-x)h(x)+r(x)$, so $(x^q-x)(h(x+a)-h(x))=r(x)-r(x+a)$. Since the right hand side has degree $<q$, we conclude that $h(x+a)-h(x)=r(x+a)-r(x)=0$. $r(x)$ is then translation invariant and therefore constant, for its degree is less than $q$, and $h$ is also translation invariant of degree $\deg(f)-q$. By induction, there is $t\in k[x]$ such that $h(x)=t(x^q-x)$. So we take $g(x)=xt(x)+r$.
\end{proof}

Let $f\in k[x]$ be translation invariant, and $g\in k[x]$ of degree $d$ such that $f(x)=g(x^q-x)$. Let $\psi:k\to\QQ^\star$ be a non-trivial additive character. The Artin-Scheier-reduced degree of $f$ (i.e. the lowest degree of a polynomial which is Artin-Schreier equivalent to $f$) is $q(d-1)+1$ (since $g(x^q-x)=a_dx^{qd}+da_dx^{q(d-1)+1}+(\text{terms of degree }\leq q(d-1))$). Therefore the Weil bound for exponential sums gives
$$
\left|\sum_{x\in k_r}\psi(\Tr_{k_r/k}(f(x)))\right|\leq q(d-1)q^{\frac{r}{2}}=(d-1)q^{\frac{r}{2}+1}
$$

On the other hand, since $f(x)=g(x^q-x)$ we get, for every $r\geq 1$,
$$
\sum_{x\in k_r}\psi(\Tr_{k_r/k}(f(x)))=\sum_{x\in k_r}\psi(\Tr_{k_r/k}(g(x^q-x)))=
$$
$$
=\sum_{t\in k_r}\#\{x\in k_r|x^q-x=t\}\psi(\Tr_{k_r/k}(g(t)))=
$$
$$
=\sum_{t\in k_r}\sum_{u\in k}\psi(u\Tr_{k_r/k}(t))\psi(\Tr_{k_r/k}(g(t)))
=\sum_{u\in k}\sum_{t\in k_r}\psi(\Tr_{k_r/k}(g(t)+ut)).
$$

Consider the $\QQ$-sheaf $\LL_{\psi(g)}:=g^\star\LL_\psi$ on $\AAA^1_k$, where $\LL_\psi$ is the Artin-Schreier sheaf associated to $\psi$. The Fourier transform of the object $\LL_{\psi(g)}[1]$ with respect to $\psi$ \cite{katz1985transformation} is a single sheaf $\FF_g$ placed in degree $-1$. The sheaf $\FF_g$ is irreducible and smooth of rank $d-1$ on $\AAA^1_k$, and totally wild at infinity with a single slope $\frac{d}{d-1}$ and Swan conductor $d$ \cite[Theorem 17]{katz-monodromy}. We have
$$
\sum_{x\in k_r}\psi(\Tr_{k_r/k}(f(x)))=\sum_{x\in k_r}\psi(\Tr_{k_r/k}(g(x^q-x)))
=\sum_{u\in k}\sum_{t\in k_r}\psi(\Tr_{k_r/k}(g(t)+ut))=
$$
\begin{equation}\label{adams}
=-\sum_{u\in k}\Tr(\Frob^r_{k,u}|(\FF_g)_u)=-\sum_{u\in k}\Tr(\Frob_{k,u}|[\FF_g]^r_u)
\end{equation}
where $[\FF_g]^r$ is the $r$-th Adams power of $\FF_g$ \cite{fu2004moment}.

Let $g(x)=\sum_{i=0}^d a_ix^i$. Recall the following facts about the local and global monodromies of the sheaf $\FF_g$:

\begin{enumerate}
 \item Suppose that $p>d$ and $k$ contains all $2(d-1)$-th roots of $-da_d$. Let $u(t)=\sum_{i\geq 0} r_i t^{1-i}\in tk[[t^{-1}]]$ be a power series such that $f'(t)+u(t)^{d-1}=0$ and let $v(t)=\sum_{i\geq 0} s_i t^{1-i}$ be the inverse image of $t$ under the automorphism $k((t^{-1}))\to k((t^{-1}))$ defined by $t^{-1}\mapsto u(t)^{-1}$ (cf. \cite[Proposition 3.1]{fu2007calculation}). Let $h(t)=\sum_{i=0}^d b_it^i$ be the polynomial obtained from $f(v(t))+v(t)t^{d-1}\in t^dk[[t^{-1}]]$ by removing the terms with negative exponent. Then, as a representation of the decomposition group $D_\infty$ at infinity, we have
$$
\FF_g\cong [d-1]_\star(\LL_{\psi(h(t))}\otimes\LL_{\rho^d(s_0t)})\otimes\rho(d(d-1)a_d/2)^{deg}\otimes g(\rho,\psi)^{deg}
$$
where $\rho:k^\star\to\QQ^\star$ is the quadratic character, $g(\rho,\psi)=-\sum_{t\in k}\rho(t)\psi(t)$ the corresponding Gauss sum and $[d-1]_\star:\GG_{m,k}\to\GG_{m,k}$ the $(d-1)$-th power map \cite[Equation 3]{haessig-lospotgafoes}. Notice that $s_0^{d-1}=-1/da_d$.

\item Suppose that $p>2$, and let $G\subseteq \mathrm{GL}(V)$ be the geometric monodromy group of $\FF_g$, where $V$ is its stalk at a geometric generic point. Then by \cite[Propositions 11.1 and 11.6]{such2000monodromy}, either $G$ is finite or $G_0$ (the unit connected component of $G$) is $\mathrm{SL}(V)$ or $\mathrm{Sp}(V)$ in its standard representation. By \cite[proof of Theorem 19]{katz-monodromy}, for $p>d$ the $\mathrm{Sp}$ case occurs if and only if $g(x+c)+d$ is odd for some $c,d\in k$. Moreover for $p>2d-1$ $G$ is never finite by \cite[Theorem 19]{katz-monodromy}. See \cite[Section 2]{haessig-lospotgafoes} for some other criterions that rule out the finite monodromy case in the $p\leq 2d-1$ case.
\end{enumerate}

The determinant of $\FF_g$ is computed over $\bar k$ in \cite[Theorem 17]{katz-monodromy}. In order to obtain a good estimate in the exceptional case below, we need to find its value over $k$.

\begin{lem}\label{determinant}
 Suppose that $p>d$ and $k$ contains all $2(d-1)$-th roots of $-da_d$. Then
$$
\det\FF_g\cong\LL_{\psi((d-1)b_{d-1}t+(d-1)b_0)}\otimes\rho^d(-1)^{deg}\otimes\rho^{d-1}(d(d-1)a_d/2)^{deg}\otimes (g(\rho,\psi)^{d-1})^{deg}
$$
\end{lem}

\begin{proof}
 Note that the result is compatible with \cite[Theorem 17]{katz-monodromy}, since $b_{d-1}=a_{d-1}s_0^{d-1}=a_{d-1}/r_0^{d-1}=-a_{d-1}/da_d$ as one can easily check.

 Let $D_\infty^{d-1}\subseteq D_\infty$ be the closed subgroup of index $d-1$ which fixes $1/t^{d-1}$. Since $k$ contains all $(d-1)$-th roots of unity, $D_\infty^{d-1}$ is normal in $D_\infty$ and the quotient $D_\infty/D_\infty^{d-1}$ is generated by $t\mapsto \zeta t$, where $\zeta\in k$ is a primitive $(d-1)$-th root of unity. Using the previous description of the representation of $D_\infty$ given by $\FF_g$, we get an isomorphism of $D^{d-1}_\infty$-representations
$$
[d-1]^\star\FF_g\cong
$$
$$
\cong\left(\bigoplus_{i=0}^{d-2}(t\mapsto\zeta^i t)^\star\LL_{\psi(h(t))}\otimes\LL_{\rho^d(s_0t)}\right)\otimes\rho(d(d-1)a_d/2)^{deg}\otimes g(\rho,\psi)^{deg}\cong
$$
$$
\cong\left(\bigoplus_{i=0}^{d-2}\LL_{\psi(h(\zeta^i t))}\otimes\LL_{\rho^d(s_0 \zeta^it)}\right)\otimes\rho(d(d-1)a_d/2)^{deg}\otimes g(\rho,\psi)^{deg}
$$
so
$$
[d-1]^\star\det\FF_g\cong\det[d-1]^\star\FF_g\cong
$$
$$
\cong
\left(\bigotimes_{i=0}^{d-2}\LL_{\psi(h(\zeta^i t))}\otimes\LL_{\rho^d(s_0 \zeta^it)}\right)\otimes\rho^{d-1}(d(d-1)a_d/2)^{deg}\otimes (g(\rho,\psi)^{d-1})^{deg}\cong
$$
$$
\cong\LL_{\psi(\sum_{i=0}^{d-2} h(\zeta^i t))}\otimes\LL_{\rho^d(\prod_{i=0}^{d-2}(s_0\zeta^it))}\otimes\rho^{d-1}(d(d-1)a_d/2)^{deg}\otimes (g(\rho,\psi)^{d-1})^{deg}\cong
$$
$$
\cong\LL_{\psi((d-1)b_{d-1}t^{d-1}+(d-1)b_0)}\otimes\LL_{\rho^d((-1)^d(s_0t)^{d-1})}\otimes\rho^{d-1}(d(d-1)a_d/2)^{deg}\otimes (g(\rho,\psi)^{d-1})^{deg}\cong
$$
$$
\cong\LL_{\psi((d-1)b_{d-1}t^{d-1}+(d-1)b_0)}\otimes\LL_{\rho^{d(d-1)}(-s_0t)}\otimes\rho^d(-1)^{deg}\otimes\rho^{d-1}(d(d-1)a_d/2)^{deg}\otimes (g(\rho,\psi)^{d-1})^{deg}\cong
$$
$$
\cong\LL_{\psi((d-1)b_{d-1}t^{d-1}+(d-1)b_0)}\otimes\rho^d(-1)^{deg}\otimes\rho^{d-1}(d(d-1)a_d/2)^{deg}\otimes (g(\rho,\psi)^{d-1})^{deg}
$$
since $\sum_{i=0}^{d-2}(\zeta^j)^i=0$ for $(d-1)\not | j$, $d(d-1)$ is even and $\prod_{i=0}^{d-2}\zeta^i=(-1)^d$.

In particular, $[d-1]^\star(\det\FF_g)$ and $$[d-1]^\star\LL_{\psi((d-1)b_{d-1}t+(d-1)b_0)}\otimes\rho^d(-1)^{deg}\otimes\rho^{d-1}(d(d-1)a_d/2)^{deg}\otimes (g(\rho,\psi)^{d-1})^{deg}$$ are isomorphic characters of $D_\infty^{d-1}$, so there is some character $\chi:k^\star\to\QQ^\star$ with $\chi^{d-1}={\mathbf 1}$ such that
$$
 \det\FF_g\cong\LL_\chi\otimes\LL_{\psi((d-1)b_{d-1}t+(d-1)b_0)}\otimes\rho^d(-1)^{deg}\otimes\rho^{d-1}(d(d-1)a_d/2)^{deg}\otimes (g(\rho,\psi)^{d-1})^{deg}
$$
as representations of $D_\infty$. But then
$$
 \widehat{(\det\FF_g)}\otimes\LL_\chi\otimes\LL_{\psi((d-1)b_{d-1}t+(d-1)b_0)}\otimes\rho^d(-1)^{deg}\otimes\rho^{d-1}(d(d-1)a_d/2)^{deg}\otimes (g(\rho,\psi)^{d-1})^{deg}
$$
is a rank 1 smooth sheaf on $\GG_{m,k}$, tamely ramified at $0$ and unramified at infinity, so it must be geometrically trivial, that is, $\chi$ is trivial (since everything else is unramified at $0$). Moreover, since the Frobenius action is trivial at infinity it must be the trivial sheaf. Therefore
$$
 \det\FF_g\cong\LL_{\psi((d-1)b_{d-1}t+(d-1)b_0)}\otimes\rho^d(-1)^{deg}\otimes\rho^{d-1}(d(d-1)a_d/2)^{deg}\otimes (g(\rho,\psi)^{d-1})^{deg}
$$
as sheaves on $\AAA^1_k$.
\end{proof}

\begin{prop}\label{SL}
 Suppose that $p>d$, the sheaf $\FF_g$ does not have finite monodromy (e.g. $p>2d-1$) and there do not exist $c,d\in k$ such that $g(x+c)+d$ is odd. Then we have an estimate
$$
\left|\sum_{x\in k_r}\psi(\Tr_{k_r/k}(f(x)))\right|\leq C_{d,r}q^{\frac{r+1}{2}}
$$
where
$$
C_{d,r}=\frac{1}{d-1}\sum_{i=0}^{d-1}|i-1|{{d-2+r-i}\choose{r-i}}{{d-1}\choose i}
$$
unless $a_{d-1}=0$ and $r=d-1$, in which case there is an estimate
$$
\left|\sum_{x\in k_r}\psi(\Tr_{k_r/k}(f(x)))-(-1)^{d-1}q\cdot\rho^d(-1)(\psi(b_0)\rho(d(d-1)a_d/2)g(\rho,\psi))^{d-1}\right|< C_{d,r}q^{\frac{r+1}{2}}.
$$
\end{prop}

\begin{proof}
 By \cite[Section 1]{fu2004moment}, we have
$$
\sum_{x\in k_r}\psi(\Tr_{k_r/k}(f(x)))=-\sum_{u\in k}\Tr(\Frob_{k,u}|[\FF_g]^r_u)=
$$
$$
=\sum_{i=0}^r(-1)^{i-1}(i-1)\Tr(\Frob_k,\HH^1_c(\AAA^1_{\bar k},\Sym^{r-i}\FF_g\otimes\wedge^i\FF_g))-
$$
$$
-\sum_{i=0}^r(-1)^{i-1}(i-1)\Tr(\Frob_k,\HH^2_c(\AAA^1_{\bar k},\Sym^{r-i}\FF_g\otimes\wedge^i\FF_g)).
$$

Let $G\subseteq\mathrm{GL}(V)$ be the geometric monodromy group of $\FF_g$. Under the hypotheses of the proposition, the unit connected component of $G$ is $\mathrm{SL}(V)$, so $G$ is the inverse image of its image by the determinant. By lemma \ref{determinant}, $G$ is $\mathrm{SL}(V)$ if $b_{d-1}=0$ (if and only if $a_{d-1}=0$) and $\mathrm{GL}_p(V)=\mu_p\cdot\mathrm{SL}(V)$ (since $p>d$, so $p$ does not divide $d-1$) if $b_{d-1}\neq 0$.

For every $i$, the dimension of $\HH^2_c(\AAA^1_{\bar k},\Sym^{r-i}\FF_g\otimes\wedge^i\FF_g)$ is the dimension of the coinvariant (or the invariant) space of the action of $G$ on $\Sym^{r-i}V\otimes\wedge^i V$. By \cite[Corollary 4.2]{rojas-leon-biotwbfac}, the action of $\mathrm{SL}(V)\subseteq G$ on $\Sym^{r-i}V\otimes\wedge^i V$ has no invariants unless $r=d-1$ and $i=r,r-1$, in which case the invariant space $W_i$ is one-dimensional. If $a_{d-1}\neq 0$, a generator $\zeta_p$ of the quotient $G/\mathrm{SL}(V)\cong\mu_p$ acts on $W_i$ via multiplication by $\zeta_p^{d-1}$, which can not be trivial since $p>d$. So the action of $G$ has no invariants on $\Sym^{r-i}V\otimes\wedge^i V$ for any $i$ if $a_{d-1}\neq 0$.

In that case, since $\HH^1_c(\AAA^1_{\bar k},\Sym^{r-i}\FF_g\otimes\wedge^i\FF_g)$ is mixed of weights $\leq r+1$ we get
$$
\left|\sum_{x\in k_r}\psi(\Tr_{k_r/k}(f(x)))\right|\leq\sum_{i=0}^r|i-1|\dim\HH^1_c(\AAA^1_{\bar k},\Sym^{r-i}\FF_g\otimes\wedge^i\FF_g)\cdot q^{\frac{r+1}{2}}.
$$
Moreover, by the Ogg-Shafarevic formula we have
$$
\dim\HH^1_c(\AAA^1_{\bar k},\Sym^{r-i}\FF_g\otimes\wedge^i\FF_g)=-\chi(\AAA^1_{\bar k},\Sym^{r-i}\FF_g\otimes\wedge^i\FF_g)=
$$
$$
=\mathrm{Swan}_\infty(\Sym^{r-i}\FF_g\otimes\wedge^i\FF_g)-\mathrm{rank}(\Sym^{r-i}\FF_g\otimes\wedge^i\FF_g)\leq
$$
$$
\leq\frac{1}{d-1}\mathrm{rank}(\Sym^{r-i}\FF_g\otimes\wedge^i\FF_g)=\frac{1}{d-1}{{d-2+r-i}\choose{r-i}}{{d-1}\choose i}
$$
since all slopes at infinity of $\FF_g$ (and a fortiori of $\Sym^{r-i}\FF_g\otimes\wedge^i\FF_g$) are $\leq \frac{d}{d-1}$.

Suppose now that $a_{d-1}=0$ and $r=d-1$. As in \cite[Corollary 4.2]{rojas-leon-biotwbfac}, we have 
$$
\sum_{i=0}^r(-1)^{i-1}(i-1)\Tr(\Frob_k,\HH^2_c(\AAA^1_{\bar k},\Sym^{r-i}\FF_g\otimes\wedge^i\FF_g))=
$$
$$
=(-1)^{r}(r-2)\Tr(\Frob_k,\HH^2_c(\AAA^1_{\bar k},\Sym^{1}\FF_g\otimes\wedge^{r-1}\FF_g))+
$$
$$
+(-1)^{r-1}(r-1)\Tr(\Frob_k,\HH^2_c(\AAA^1_{\bar k},\wedge^r\FF_g))=
$$
$$
=(-1)^{r-1}\Tr(\Frob_k,\HH^2_c(\AAA^1_{\bar k},\det\FF_g))=
$$
$$
=(-1)^dq\cdot\psi((d-1)b_0)\rho^d(-1)\rho^{d-1}(d(d-1)a_d/2)g(\rho,\psi)^{d-1}=
$$
$$
=(-1)^{d}q\cdot\rho^d(-1)(\psi(b_0)\rho(d(d-1)a_d/2)g(\rho,\psi))^{d-1}
$$
by lemma \ref{determinant}. We conclude as above using the fact that, for the two values of $i$ for which $\HH^2_c(\AAA^1_{\bar k},\Sym^{r-i}\FF_g\otimes\wedge^i\FF_g)$ is one-dimensional, the sheaf $\Sym^{r-i}\FF_g\otimes\wedge^i\FF_g$ has at least one slope equal to $0$ at infinity, so
$$
\dim\HH^1_c(\AAA^1_{\bar k},\Sym^{r-i}\FF_g\otimes\wedge^i\FF_g)=1-\chi(\AAA^1_{\bar k},\Sym^{r-i}\FF_g\otimes\wedge^i\FF_g)=
$$
$$
=1+\mathrm{Swan}_\infty(\Sym^{r-i}\FF_g\otimes\wedge^i\FF_g)-\mathrm{rank}(\Sym^{r-i}\FF_g\otimes\wedge^i\FF_g)\leq
$$
$$
\leq 1+\frac{d}{d-1}(\mathrm{rank}(\Sym^{r-i}\FF_g\otimes\wedge^i\FF_g)-1)-\mathrm{rank}(\Sym^{r-i}\FF_g\otimes\wedge^i\FF_g)<
$$
$$
<\frac{1}{d-1}\mathrm{rank}(\Sym^{r-i}\FF_g\otimes\wedge^i\FF_g)=\frac{1}{d-1}{{d-2+r-i}\choose{r-i}}{{d-1}\choose i}.
$$
\end{proof}

\begin{prop}
 Suppose that $p>d$, the sheaf $\FF_g$ does not have finite monodromy (e.g. $p>2d-1$) and there exist $\alpha,\beta\in k$ such that $g(x+\alpha)+\beta$ is odd (so $d$ is odd). Then we have an estimate
$$
\left|\sum_{x\in k_r}\psi(\Tr_{k_r/k}(f(x)))\right|\leq C_{d,r}q^{\frac{r+1}{2}}
$$
where
$$
C_{d,r}=\frac{1}{d-1}\sum_{i=0}^{d-1}|i-1|{{d-2+r-i}\choose{r-i}}{{d-1}\choose i}
$$
unless $a_{d-1}=0$ and $r\leq d-1$ is even, in which case there is an estimate
$$
\left|\sum_{x\in k_r}\psi(\Tr_{k_r/k}(f(x)))-(-1)^r\psi(-\beta)^rq^{\frac{r}{2}+1}\right|< C_{d,r}q^{\frac{r+1}{2}}.
$$
\end{prop}

\begin{proof}
The proof is similar to the previous one. In this case, the unit connected component of $G$ is $\mathrm{Sp}(V)$, so by lemma \ref{determinant} $G$ is $\mathrm{Sp}(V)$ if $b_{d-1}=0$ (if and only if $a_{d-1}=0$) and $\mu_p\cdot\mathrm{SL}(V)$ (since $p>d$, so $p$ does not divide $d-1$) if $b_{d-1}\neq 0$.

By \cite[lemma on p.62]{katz2001frobenius}, the action of $\mathrm{Sp}(V)\subseteq G$ on $\Sym^{r-i}V\otimes\wedge^i V$ has no invariants unless $r\leq d-1$ is even and $i=r,r-1$, in which case the invariant space $W_i$ is one-dimensional. If $a_{d-1}\neq 0$, a generator $\zeta_p$ of the quotient $G/\mathrm{Sp}(V)\cong\mu_p$ acts on $W_i$ via multiplication by $\zeta_p^{d-1}$, which can not be trivial since $p>d$. So the action of $G$ has no invariants on $\Sym^{r-i}V\otimes\wedge^i V$ for any $i$ if $a_{d-1}\neq 0$. We conclude this case as in the previous proposition.

Suppose now that $a_{d-1}=0$, $r\leq d-1$ is even and $i=r$ or $r-1$. Since the coefficient of $x^{d-1}$ in $g(x)$ is $0$, the coefficient in $g(x+\alpha)+\beta$ is $da_d\alpha$, so it can only be an odd polynomial if $\alpha=0$. That is, $g(x)+\beta$ is odd, or equivalently, $g(-x)=-2\beta-g(x)$. Then the sheaf $\psi(\beta)^{deg}\otimes\FF_g(1/2)$ is self-dual: since the dual of $\LL_{\psi(g)}$ is $\LL_{\psi(-g)}(1)$, using that $D\circ FT_\psi=[-1]^\star FT_\psi\circ D(1)$ \cite[Corollaire 2.1.5]{katz1985transformation} we get that the dual of $\FF_g=\HHH^{-1}(FT_\psi(\LL_{\psi(g)}[1]))$ is
$$
[-1]^\star \HHH^{-1}(FT_\psi\LL_{\psi(-g)}(1))=[-1]^\star\FF_{-g}(1)=[-1]^\star\FF_{2\beta+g(-x)}(1)=\psi(2\beta)^{deg}\otimes\FF_g(1)
$$
so $\psi(\beta)^{deg}\otimes\FF_g(1/2)$ is self-dual (symplectically, since it is so geometrically by \cite[Theorem 19]{katz-monodromy}). In particular, the one-dimensional $\mathrm{Sp}(V)$-invariant subspace of $(\Sym^{r-i}\FF_g\otimes\wedge^i\FF_g)\otimes\psi(\beta)^{r\cdot deg}(r/2)$ is also invariant under all Frobenii. So $W_i$ is in fact the geometrically constant sheaf $\psi(-\beta)^{r\cdot deg}(-r/2)$. In particular
$$
\sum_{i=0}^r(-1)^{i-1}(i-1)\Tr(\Frob_k,\HH^2_c(\AAA^1_{\bar k},\Sym^{r-i}\FF_g\otimes\wedge^i\FF_g))=
$$
$$
=(-1)^{r}(r-2)\Tr(\Frob_k,\HH^2_c(\AAA^1_{\bar k},\Sym^{1}\FF_g\otimes\wedge^{r-1}\FF_g))+
$$
$$
+(-1)^{r-1}(r-1)\Tr(\Frob_k,\HH^2_c(\AAA^1_{\bar k},\wedge^r\FF_g))=
$$
$$
=(-1)^{r-1}\Tr(\Frob_k,\HH^2_c(\AAA^1_{\bar k},\psi(-\beta)^{r\cdot deg}(-r/2)))=(-1)^{r-1}\psi(-\beta)^rq^{\frac{r}{2}+1}.
$$
We conclude as in the previous proposition.
\end{proof}

\section{Multiplicative character sums for translation invariant polynomials}

Let $f\in k[x]$ be translation invariant, and $g\in k[x]$ of degree $d$ such that $f(x)=g(x^q-x)$. Let $\chi:k^\star\to\QQ^\star$ a non-trivial multiplicative character of order $m$, extended by zero to all of $k$. Since $f$ has degree $qd$, Weil's bound gives in this case
$$
\left|\sum_{x\in k_r}\chi(\Nm_{k_r/k}(f(x)))\right|\leq (qd-1)q^{\frac{r}{2}}.
$$
On the other hand we have, for every $r\geq 1$,
$$
\sum_{x\in k_r}\chi(\Nm_{k_r/k}(f(x)))=\sum_{x\in k_r}\chi(\Nm_{k_r/k}(g(x^q-x)))=
$$
$$
=\sum_{t\in k_r}\#\{x\in k_r|x^q-x=t\}\chi(\Nm_{k_r/k}(g(t)))=
$$
$$
=\sum_{t\in k_r}\sum_{u\in k}\psi(u\Tr_{k_r/k}(t))\chi(\Nm_{k_r/k}(g(t)))
=\sum_{u\in k}\sum_{t\in k_r}\psi(u\Tr_{k_r/k}(t))\chi(\Nm_{k_r/k}(g(t))).
$$

Consider the $\QQ$-sheaf $\LL_{\chi(g)}:=g^\star\LL_\chi$ on $\AAA^1_k$, where $\LL_\chi$ is the Kummer sheaf on $\GG_{m,k}$ associated to $\chi$ \cite[1.7]{deligne569application}, extended by zero to $\AAA^1_k$. Suppose that $g$ is square-free and its degree $d$ is prime to $p$. Then $\LL_{\chi(g)}$ is an irreducible middle extension sheaf, smooth on the complement of the subscheme $Z\subseteq\AAA^1_k$ defined by $g=0$. Since there is at least one point where it is not smooth, it is not isomorphic to an Artin-Schreier sheaf and therefore the Fourier transform of $\LL_{\chi(g)}[1]$ is a single irreducible middle extension sheaf $\FF_g$ placed in degree $-1$ \cite[8.2]{katz1988gsk}. We have
$$
\sum_{x\in k_r}\chi(\Nm_{k_r/k}(f(x)))=\sum_{x\in k_r}\chi(\Nm_{k_r/k}(g(x^q-x)))
=\sum_{u\in k}\sum_{t\in k_r}\psi(u\Tr_{k_r/k}(t))\chi(\Nm_{k_r/k}(g(t)))=
$$
\begin{equation}\label{adams2}
=-\sum_{u\in k}\Tr(\Frob^r_{k,u}|(\FF_g)_u)=-\sum_{u\in k}\Tr(\Frob_{k,u}|[\FF_g]^r_u)
\end{equation}
where $[\FF_g]^r$ is the $r$-th Adams power of $\FF_g$.

\begin{prop}\label{monoinf}
 The sheaf $\FF_g$ has generic rank $d$, it is smooth on $\GG_{m,k}$ and tamely ramified at $0$. Its rank at $0$ is $d-1$. If all roots of $g$ are in $k$, the action of the decomposition group $D_\infty$ on the generic stalk of $\FF_g$ splits as a direct sum $\bigoplus_a \chi(g'(a))^{deg}\otimes g(\chi,\psi)^{deg}\otimes\LL_{\bar\chi}\otimes\LL_{\psi_a}$ where the sum is taken over the roots of $f$, $\LL_{\psi_a}$ is the Artin-Schreier sheaf corresponding to the character $t\mapsto \psi(at)$ and $g(\chi,\psi)=-\sum_t\chi(t)\psi(t)$ if the Gauss sum.  
\end{prop}
 
\begin{proof}
 The generic rank of $\FF_g$ is the dimension of $\HH^1_c(\AAA^1_{\bar k},\LL_{\chi(g)}\otimes\LL_{\psi_z})$ for generic $z$. Since $\LL_{\chi(g)}$ is tamely ramified everywhere and has rank one, for any $z\neq 0$ $\LL_{\chi(g)}\otimes\LL_{\psi_z}$ is tamely ramified at every point of $\AAA^1_{\bar k}$ and totally wild at infinity with Swan conductor $1$. In particular its $\HH^i_c$ vanish for $i\neq 1$. By the Ogg-Shafarevic formula, its Euler characteristic is then $1-d-1=-d$, since there are $d$ points in $\AAA^1_{\bar k}$ where the stalk is zero. Therefore $\dim\HH^1_c(\AAA^1_{\bar k},\LL_{\chi(g)}\otimes\LL_{\psi_z})=d$ for every $z\neq 0$. Similarly, it is $d-1$ for $z=0$. Since $\FF_g$ is a middle extension, it is smooth exactly on the open set where the rank is maximal, so it is smooth on $\GG_{m,k}$. It is tamely ramified at zero, since $\LL_{\chi(g)}$ is tamely ramified at infinity \cite[Th\'eor\`eme 2.4.3]{laumon1987transformation}.

Suppose now that all roots of $g$ are in $k$, and let $a$ be one such root. In an \'etale neighborhood of $a$, the sheaf $\LL_{\chi(g)}$ is isomorphic to $\LL_{\chi(g'(a)(x-a))}=\chi(g'(a))^{deg}\otimes\LL_{\chi(x-a)}$, since $g(x)=g'(a)(x-a)\frac{g(x)}{g'(a)(x-a)}$ and $\frac{g(x)}{g'(a)(x-a)}$ is an $m$-th power in the henselization of $\AAA^1_k$ at $a$ (since its image in the residue field is $1$). Applying Laumon's local Fourier transform \cite[Proposition 2.5.3.1]{laumon1987transformation} and using that Fourier transform commutes with tensoring by unramified sheaves, we deduce that the $D_\infty$-representation $\FF_g$ contains $(LFT_\psi^{(0,\infty)}\chi(g'(a))^{deg}\otimes\LL_{\chi})\otimes\LL_{\psi_a}=\chi(g'(a))^{deg}\otimes g(\chi,\psi)^{deg}\otimes\LL_{\bar\chi}\otimes\LL_{\psi_a}$ as a direct summand. Since $g$ has $d$ distinct roots we obtain $d$ different terms this way, which is the rank of $\FF_g$, so its monodromy at $\infty$ is the direct sum of these terms. 
\end{proof}

Define by induction the sequence of polynomials $g_n[x]\in k[x]$ for $n\geq 1$ by: $g_1(x)=g(x)$, and for $n\geq 1$ $g_{n+1}(x)$ is the resultant in $t$ of $g_n(t)$ and $g(x-t)$. 

\begin{cor}\label{estimatemult}
 Suppose that either $m$ does not divide $r$ or $g_r(0)\neq 0$. Then we have an estimate
$$
\left|\sum_{x\in k_r}\chi(\Nm_{k_r/k}(f(x)))\right|\leq C_{d,r}q^{\frac{r+1}{2}}
$$
where
$$
C_{d,r}=\sum_{i=0}^{r}|i-1|\left({{d-1+r-i}\choose{r-i}}{{d}\choose i}-{{d-2+r-i}\choose{r-i}}{{d-1}\choose i}\right).
$$
\end{cor}

\begin{proof}
By the previous proposition, the action of the inertia group $I_\infty$ on $\FF_g^{\otimes r}$ splits as a direct sum over the $r$-uples of roots of $f$ 
$$
\bigoplus_{(a_1,\ldots,a_r)}\LL^{\otimes r}_{\bar\chi}\otimes\LL_{\psi_{a_1}}\otimes\cdots\otimes\LL_{\psi_{a_r}}=
\bigoplus_{(a_1,\ldots,a_r)}\LL^{\otimes r}_{\bar\chi}\otimes\LL_{\psi_{a_1+\cdots +a_r}}.
$$

For each $(a_1,\ldots,a_r)$, the character $\LL^{\otimes r}_{\bar\chi}\otimes\LL_{\psi_{a_1+\cdots +a_r}}$ is trivial if and only if both $\LL^{\otimes r}_{\bar\chi}$ and $\LL_{\psi_{a_1+\cdots +a_r}}$ are trivial, that is, if and only if $m$ divides $r$ and $a_1+\cdots+a_r=0$. Under the hypotheses of the corollary, at least one of these conditions does not hold (since the sums $a_1+\cdots+a_r$ are the roots of $g_r$). So $\FF_g^{\otimes r}$ has no invariants under the action of $I_\infty$ and, a fortiori, under the action of the larger group $\pi_1(\GG_{m,\bar k},\bar\eta)$. Since $\Sym^{r-i}\FF_g\otimes\wedge^i\FF_g$ is a subsheaf of $\FF_g^{\otimes r}$ for every $i$, we conclude that $\HH^2_c(\AAA^1_{\bar k},\Sym^{r-i}\FF_g\otimes\wedge^i\FF_g)=0$ for every $i=0,\ldots,r$. Therefore
$$
\sum_{x\in k_r}\chi(\Nm_{k_r/k}(f(x)))=-\sum_{u\in k}\Tr(\Frob_{k,u}|[\FF_g]^r_u)=
$$
$$
=\sum_{i=0}^r(-1)^{i-1}(i-1)\Tr(\Frob_k,\HH^1_c(\AAA^1_{\bar k},\Sym^{r-i}\FF_g\otimes\wedge^i\FF_g)).
$$

Since $\HH^1_c(\AAA^1_{\bar k},\Sym^{r-i}\FF_g\otimes\wedge^i\FF_g)$ is mixed of weights $\leq r+1$, we get
$$
\left|\sum_{x\in k_r}\chi(\Nm_{k_r/k}(f(x)))\right|\leq\sum_{i=0}^r|i-1|\dim\HH^1_c(\AAA^1_{\bar k},\Sym^{r-i}\FF_g\otimes\wedge^i\FF_g)\cdot q^{\frac{r+1}{2}}.
$$
And by the Ogg-Shafarevic formula, we have
$$
\dim\HH^1_c(\AAA^1_{\bar k},\Sym^{r-i}\FF_g\otimes\wedge^i\FF_g)=-\chi(\AAA^1_{\bar k},\Sym^{r-i}\FF_g\otimes\wedge^i\FF_g)=
$$
$$
=\mathrm{Swan}_\infty(\Sym^{r-i}\FF_g\otimes\wedge^i\FF_g)-\mathrm{rank}_0(\Sym^{r-i}\FF_g\otimes\wedge^i\FF_g)\leq
$$
$$
\leq{{d-1+r-i}\choose{r-i}}{{d}\choose i}-{{d-2+r-i}\choose{r-i}}{{d-1}\choose i}
$$
by the previous proposition, since $\FF_g$ is smooth on $\GG_{m,k}$, tamely ramified at $0$ and all its slopes at infinity (and thus all slopes of of $\Sym^{r-i}\FF_g\otimes\wedge^i\FF_g$) are $\leq 1$.

\end{proof}

\begin{cor}\label{det2}
 If all roots of $g(x)=\sum_{i=0}^da_ix^i$ are in $k$, the determinant of $\FF_g$ is $\chi((-1)^{d(d-1)/2}a_d^{-(d-2)}\mathrm{disc}(g))^{deg}\otimes (g(\chi,\psi)^d)^{deg}\otimes\LL_{\bar\chi^d}\otimes\LL_{\psi_{-a_{d-1}/a_d}}$.
\end{cor}

\begin{proof}
 By proposition \ref{monoinf}, the action of $D_\infty$ on the determinant of $\FF_g$ is given by 
$$
\GGG:=\bigotimes_a \chi(g'(a))^{deg}\otimes g(\chi,\psi)^{deg}\otimes\LL_{\bar\chi}\otimes\LL_{\psi_a}=\chi(\prod_a g'(a))^{deg}\otimes(g(\chi,\psi)^d)^{deg}\otimes\LL_{\bar\chi^d}\otimes\LL_{\psi_{\sum a}}
$$
where the product is taken over the roots of $g$. Now $\sum a=-a_{d-1}/a_d$, and $$\prod_a g'(a)=\prod_a a_d\prod_{g(b)=0,b\neq a}(b-a)=
$$
$$
=a_d^d\prod_{g(a)=g(b)=0,a\neq b}(a-b)=(-1)^{d(d-1)/2}a_d^{-(d-2)}\mathrm{disc}(g).$$

Therefore $\det(\FF_g)\otimes\hat\GGG$ is smooth on $\GG_{m,k}$, tamely ramified at zero and unramified at infinity, so it is geometrically constant. Looking at the Frobenius action at $0$, it must be the constant sheaf $\QQ$. We conclude that $\det(\FF_g)\cong\GGG$. 
\end{proof}

 \begin{prop} Let $h(x)=g(x-\frac{a_{d-1}}{da_d})$. Suppose that $p>2d+1$ and $h$ is not odd (for $d$ odd) or even (for $d$ even). Then the geometric monodromy group $G$ of $\FF_g$ is $\mathrm{GL}_{sp}(V)$ if $a_{d-1}\neq 0$ and $\mathrm{GL}_s(V)$ if $a_{d-1}=0$, where $V$ is the geometric generic stalk of $\FF_g$ and $s$ is the order of $\chi^d$.
\end{prop}

\begin{proof}
 Since $\LL_{\chi(g)}$ is the translate of $\LL_{\chi(h)}$ by $a:=\frac{a_{d-1}}{da_d}$, we have $\FF_g=\FF_h\otimes\LL_{\psi_a}$. If $G$ (respectively $G'$) is the geometric monodromy group of $\FF_g$ (resp. $\FF_h$), we have then $G\subseteq \mu_p\cdot G'$ and $G'\subseteq \mu_p\cdot G$. In particular, the unit connected components $G_0$ and $G'_0$ are the same. Since $\FF_g$ is pure, $G_0$ is a semisimple group \cite[Corollaire 1.3.9]{deligne1980conjecture}, so by \cite[Theorem 7.6.3.1]{katz1990esa}, $\FF_g$ is Lie-irreducible and $G_0$ is one of $\mathrm{SL}(V)$, $\mathrm{Sp}(V)$ (only possible if $\chi^d={\mathbf 1}$) or $\mathrm{SO}(V)$ (only possible if $\chi^d$ has order $2$). We will see that, under the given hypotheses, the last two options are not possible. 

By corollary \ref{det2}, the determinant of $\FF_h$ is geometrically isomorphic to $\LL_{\bar\chi^d}$. By \cite[Proposition 6]{katz-monodromy}, the factor group $G'/G'_0$ is cyclic of finite prime to $p$ order. In particular, there exists some prime to $p$ integer $e$ such that the geometric monodromy group of the pull-back $[e]^\star\FF_h$ is in $G'_0$, where $[e]:\GG_{m,k}\to\GG_{m,k}$ is the $e$-th power map. If $G'_0=\mathrm{Sp}(V)$ or $\mathrm{SO}(V)$, $[e]^\star\FF_h$ would then be geometrically self-dual. By proposition \ref{monoinf}, its restriction to the inertia group $I_\infty$ is the direct sum of $[e]^\star\LL_{\psi_{b}}\otimes\LL_{\bar\chi^e}$ taken over the roots $b$ of $h$. Its dual is then the direct sum on $[e]^\star\LL_{\psi_{-b}}\otimes\LL_{\chi^e}$. Given that the dual of $[e]^\star\LL_{\psi_b}$ is $[e]^\star\LL_{\psi_{-b}}$, in order for this to be self-dual as a representation of $I_\infty$ a necessary condition is that the set of roots of $h$ is symmetric with respect to $0$, that is, that $h$ is either even or odd (since it is a priori square-free).

So, if $h$ is neither even nor odd, $G_0$ is $\mathrm{SL}(V)$. Then $G$ is $\mathrm{SL}_n(V)$, where $n$ is the geometric order of the determinant of $\FF_g$. By corollary \ref{det2}, this order is $sp$ if $a_{d-1}\neq 0$ and $s$ if $a_{d-1}=0$.
\end{proof}

 \begin{cor} Let $h(x)=g(x-\frac{a_{d-1}}{da_d})$. Suppose that $p>2d+1$ and $h$ is not odd (for $d$ odd) or even (for $d$ even). Then we have an estimate
$$
\left|\sum_{x\in k_r}\chi(\Nm_{k_r/k}(f(x)))\right|\leq C_{d,r}q^{\frac{r+1}{2}}
$$
where
$$
C_{d,r}=\sum_{i=0}^{r}|i-1|\left({{d-1+r-i}\choose{r-i}}{{d}\choose i}-{{d-2+r-i}\choose{r-i}}{{d-1}\choose i}\right)
$$
unless $r=d$, $\chi^d$ is trivial and $a_{d-1}=0$, in which case there exists an $\ell$-adic unit $\beta\in\QQ$ with $|\beta|=q^{\frac{d}{2}}$ such that
$$
\left|\sum_{x\in k_r}\chi(\Nm_{k_r/k}(f(x)))-(-1)^dq\beta \right|\leq C_{d,r}q^{\frac{r+1}{2}}.
$$
If $k$ contains all roots of $g$, then $\beta=\chi((-1)^{d(d-1)/2}a_d^{-(d-2)}\mathrm{disc}(g))g(\chi,\psi)^d$.
\end{cor}

\begin{proof}
 By the previous proposition, the monodromy group $G$ of $\FF_g$ is $\mathrm{GL}_{sp}(V)$ if $a_{d-1}\neq 0$ and $\mathrm{GL}_s(V)$ if $a_{d-1}=0$. We proceed as in the proof of proposition \ref{SL}: $G_0$ has no invariants on $\Sym^{r-i}V\otimes\wedge^i V$ unless $r=d$ and $i=r,r-1$, in which case the invariant space is one-dimensional and $G$ acts on it via multiplication by the determinant. So the action of $G$ does not have invariants unless $a_{d-1}=0$ and $\chi^d$ is trivial (i.e. $m|d$) by corollary \ref{det2}. In that case we obtain the estimate as in \ref{SL}, using the value for $C_{d,r}$ computed in corollary \ref{estimatemult}.

In the exceptional case, we have again
$$
\sum_{i=0}^r(-1)^{i-1}(i-1)\Tr(\Frob_k,\HH^2_c(\AAA^1_{\bar k},\Sym^{r-i}\FF_g\otimes\wedge^i\FF_g))=
$$
$$
=(-1)^{r-1}\Tr(\Frob_k,\HH^2_c(\AAA^1_{\bar k},\det\FF_g)).
$$
Now $\det\FF_g$ is geometrically constant of weight $d$, so there exists an $\ell$-adic unit $\beta$ with $|\beta|=1$ such that $\det\FF_g=(\beta q^{\frac{d}{2}})^{deg}$. Then $\Tr(\Frob_k,\HH^2_c(\AAA^1_{\bar k},\det\FF_g))=\beta q^{\frac{d}{2}+1}$. If $k$ contains all roots of $g$, the value of $\beta$ is given in corollary \ref{det2}. We conclude as in proposition \ref{SL} using that, for the two values of $i$ for which $\HH^2_c(\AAA^1_{\bar k},\Sym^{r-i}\FF_g\otimes\wedge^i\FF_g)$ is one-dimensional, the sheaf $\Sym^{r-i}\FF_g\otimes\wedge^i\FF_g$ has at least one slope equal to $0$ at infinity, so
$$
\dim\HH^1_c(\AAA^1_{\bar k},\Sym^{r-i}\FF_g\otimes\wedge^i\FF_g)=1-\chi(\AAA^1_{\bar k},\Sym^{r-i}\FF_g\otimes\wedge^i\FF_g)=
$$
$$
=1+\mathrm{Swan}_\infty(\Sym^{r-i}\FF_g\otimes\wedge^i\FF_g)-\mathrm{rank}_0(\Sym^{r-i}\FF_g\otimes\wedge^i\FF_g)\leq
$$
$$
\leq\mathrm{gen.rank}(\Sym^{r-i}\FF_g\otimes\wedge^i\FF_g)-\mathrm{rank}_0(\Sym^{r-i}\FF_g\otimes\wedge^i\FF_g)=
$$
$$
={{d-1+r-i}\choose{r-i}}{{d}\choose i}-{{d-2+r-i}\choose{r-i}}{{d-1}\choose i}.
$$
\end{proof}

\section{Additive character sums for homothety invariant polynomials}

Let $f\in k_r[x]$ be a polynomial and $e|q-1$ an integer. Let $\Gamma_e\subseteq k^\star$ be the unique subgroup of $k^\star$ of index $e$. We say that $f$ is \emph{$\Gamma_e$-homothety invariant} if $f(\lambda x)=f(x)$ for every $\lambda\in\Gamma_e$. Equivalently, if $f(\lambda^ex)=f(x)$ for every $\lambda\in k^\star$. An argument similar to that in lemma \ref{translationinvariant} shows

\begin{lem}\label{homothetyinvariant}
 Let $f\in k_r[x]$ and $e|q-1$. The following conditions are equivalent:
\begin{itemize}
 \item[(a)] $f$ is $\Gamma_e$-homothety invariant.
 \item[(b)] There exists $g\in k_r[x]$ such that $f(x)=g(x^{\frac{q-1}{e}})$.
\end{itemize}
\end{lem}

Let $f\in k_r[x]$ be $\Gamma_e$-homothety invariant, $g\in k_r[x]$ of degree $d$ such that $f(x)=g(x^{\frac{q-1}{e}})$ and $\psi:k\to\QQ^\star$ a non-trivial additive character. Weil's bound gives in this case
$$
\left|\sum_{x\in k_r}\psi(\Tr_{k_r/k}(f(x)))\right|\leq \left(\frac{d(q-1)}{e}-1\right)q^{\frac{r}{2}}.
$$
On the other hand,
$$
\sum_{x\in k_r}\psi(\Tr_{k_r/k}(f(x)))=\psi(\Tr_{k_r/k}(f(0)))+\sum_{x\in k_r^\star}\psi(\Tr_{k_r/k}(g(x^{\frac{q-1}{e}})))=
$$
$$
=\psi(\Tr_{k_r/k}(f(0)))+\frac{q-1}{e}\sum_{\Nm_{k_r/k}(x)^e=1}\psi(\Tr_{k_r/k}(g(x)))=
$$
\begin{equation}\label{suma}
=\psi(\Tr_{k_r/k}(f(0)))+\frac{q-1}{e}\sum_{\mu^e=1}\sum_{\Nm_{k_r/k}(x)=\mu}\psi(\Tr_{k_r/k}(g(x))).
\end{equation}

For each $\mu$, we will estimate the sum $\sum_{\Nm_{k_r/k}(x)=\mu}\psi(\Tr_{k_r/k}(g(x)))$ using Weil descent. Fix a basis $\{\alpha_1,\ldots,\alpha_r\}$ of $k_r$ over $k$, and let $P(x_1,\ldots,x_r)=\prod_\sigma(\sigma(\alpha_1)x_1+\cdots+\sigma(\alpha_r)x_r)$, where the product is taken over all $\sigma\in\Gal(k_r/k)$. Since $P$ is $\Gal(k_r/k)$-invariant, its coefficients are in $k$. By construction, for every $(x_1,\ldots,x_r)\in k^r$ we have $P(x_1,\ldots,x_r)=\Nm_{k_r/k}(\alpha_1x_1+\cdots+\alpha_rx_r)$. Therefore
$$
\sum_{\Nm_{k_r/k}(x)=\mu}\psi(\Tr_{k_r/k}(g(x)))=\sum_{P(x_1,\ldots,x_r)=\mu}\psi(\Tr_{k_r/k}(g(\alpha_1x_1+\cdots+\alpha_rx_r))=
$$
$$
=\sum_{P(x_1,\ldots,x_r)=\mu}\psi\left(\sum_\sigma g^\sigma(\sigma(\alpha_1)x_1+\cdots+\sigma(\alpha_r)x_r)\right)
$$
where $g^\sigma$ is the polynomial obtained by applying $\sigma$ to the coefficients of $g$, and the sum is taken over all $r$-tuples $(x_1,\ldots,x_r)\in k^r$ such that $P(x_1,\ldots,x_r)=\mu$. By Grothendieck's trace formula, we get
\begin{equation}\label{lefschetz}
\sum_{\Nm_{k_r/k}(x)=\mu}\psi(\Tr_{k_r/k}(g(x)))=\sum_{i=0}^{2r-2}\Tr(\Frob_k|\HH^i_c(V_\mu\otimes\bar k,\LL_{\psi(G)}))
\end{equation}
where $V_\mu$ is the hypersurface defined in $\AAA^r_k$ by the equation $P(x_1,\ldots,x_r)=\mu$ and $G=\sum_\sigma g^\sigma(\sigma(\alpha_1)x_1+\cdots+\sigma(\alpha_r)x_r)\in k[x]$ (since it is $\Gal(k_r/k)$-invariant).

\begin{prop}
 Suppose that $g$ has degree $d$ prime to $p$. For any $\mu\in k^\star$, $\HH^i_c(V_\mu\otimes\bar k,\LL_{\psi(G)})=0$ for $i\neq r-1$ and $\dim\HH^{r-1}_c(V_\mu\otimes\bar k,\LL_{\psi(G)})=rd^{r-1}$.
\end{prop}

\begin{proof}
 Over $k_r$, the map $(x_1,\ldots,x_r)\mapsto (\sigma(\alpha_1)x_1+\cdots+\sigma(\alpha_r)x_r)_{\sigma\in\Gal(k_r/k)}$ is a (linear) isomorphism between $\AAA^r_{k_r}$ and $\AAA^{\Gal(k_r/k)}_{k_r}$. The pull-back of $P$ under this automorphism is just $x_1\cdots x_r$. So $V_\mu\otimes\bar k$ is isomorphic to the hypersurface $x_1\cdots x_r=\mu$, and the sheaf $\LL_{\psi(G)}$ corresponds under this isomorphism to the sheaf $\LL_{\psi(\sum_\sigma g^\sigma(x_\sigma))}=\boxtimes_\sigma\LL_{\psi(g^\sigma)}$ where $\LL_{\psi(g^\sigma)}$ is the pull-back of the Artin-Schreier sheaf $\LL_\psi$ by $g^\sigma$.

For every $\sigma\in\Gal(k_r/k)$, the sheaf $\LL_{\psi(g^\sigma)}$ is smooth on $\AAA^1_{\bar k}$ of rank one, with slope $d$ at infinity. \cite[Theorem 5.1]{katz1988gsk} shows that the class of objects of the form $\GGG[1]$ where $\GGG$ is a smooth $\QQ$-sheaf on $\GG_{m,\bar k}$, tamely ramified at $0$ and totally wild at infinity is invariant under convolution. In particular, if $m:\GG_{m,\bar k}^{\Gal(k_r/k)}\to\GG_{m,\bar k}$ is the multiplication map, $\R^im_!(\boxtimes_\sigma\LL_{\psi(g^\sigma)})=0$ for $i\neq r-1$ and $\R^{r-1}m_!(\boxtimes_\sigma\LL_{\psi(g^\sigma)})$ is smooth on $\GG_{m,\bar k}$ of rank $rd^{r-1}$, tamely ramified at $0$ and totally wild at infinity with Swan conductor $d^r$ \cite[Theorem 5.1(4,5)]{katz1988gsk}. Taking the fibre at $\mu$ proves the proposition using proper base change.
\end{proof}

\begin{cor}
 Suppose that $g$ has degree $d$ prime to $p$. Then
$$
\left|\sum_{x\in k_r^\star}\psi(\Tr_{k_r/k}(f(x)))\right|\leq rd^{r-1}(q-1)q^{\frac{r-1}{2}}
$$
\end{cor}

\begin{proof}
 Since $\LL_{\psi(G)}$ is pure of weight $0$, $\HH^{r-1}_c(V_\mu\otimes\bar k,\LL_{\psi(G)})$ is mixed of weights $\leq r-1$ for every $\mu$ (in fact it is pure of weight $r-1$ by \cite[Theorem 5.1(7)]{katz1988gsk}). So the previous proposition together with (\ref{lefschetz}) implies 
$$
\left|\sum_{\Nm_{k_r/k}(x)=\mu}\psi(\Tr_{k_r/k}(g(x)))\right|\leq rd^{r-1}q^{\frac{r-1}{2}}
$$
for every $\mu\in k^\star$. We conclude by using (\ref{suma}).
\end{proof}

\section{Multiplicative character sums for homothety invariant polynomials}

Let $e|q-1$ an integer and $f(x)=g(x^{\frac{q-1}{e}})\in k_r[x]$ $\Gamma_e$-homothety invariant as in the previous section. Let $d=\deg(g)$ and $\chi:k^\star\to\QQ^\star$ a non-trivial multiplicative characer of order $m$. Weil's bound gives
$$
\left|\sum_{x\in k_r}\chi(\Nm_{k_r/k}(f(x)))\right|\leq \left(\frac{d(q-1)}{e}-1\right)q^{\frac{r}{2}}
$$
if $g$ is not an $m$-th power. On the other hand, we have
$$
\sum_{x\in k_r}\chi(\Nm_{k_r/k}(f(x)))=\chi(\Nm_{k_r/k}(f(0)))+\sum_{x\in k_r^\star}\chi(\Nm_{k_r/k}(g(x^{\frac{q-1}{e}})))=
$$
$$
=\chi(\Nm_{k_r/k}(f(0)))+\frac{q-1}{e}\sum_{\Nm_{k_r/k}(x)^e=1}\chi(\Nm_{k_r/k}(g(x)))=
$$
\begin{equation}\label{suma2}
=\chi(\Nm_{k_r/k}(f(0)))+\frac{q-1}{e}\sum_{\mu^e=1}\sum_{\Nm_{k_r/k}(x)=\mu}\chi(\Nm_{k_r/k}(g(x))).
\end{equation}

In order to estimate the sum $\sum_{\Nm_{k_r/k}(x)=\mu}\chi(\Nm_{k_r/k}(g(x)))$, we may and will assume without loss of generality that $g(0)\neq 0$: otherwise, writing $g(x)=x^ag_0(x)$ with $g_0(0)\neq 0$, 
$$
\sum_{\Nm_{k_r/k}(x)=\mu}\chi(\Nm_{k_r/k}(g(x)))=\sum_{\Nm_{k_r/k}(x)=\mu}\chi(\Nm_{k_r/k}(x^ag_0(x)))=
$$
$$
=\sum_{\Nm_{k_r/k}(x)=\mu}\chi(\Nm_{k_r/k}(x^a))\chi(\Nm_{k_r/k}(g_0(x)))=\chi(\mu)^a\sum_{\Nm_{k_r/k}(x)=\mu}\chi(\Nm_{k_r/k}(g_0(x))),
$$ 
with $|\chi(\mu)^a|=1$.

Let $P=\prod_\sigma(\sigma(\alpha_1)x_1+\cdots+\sigma(\alpha_r)x_r)$ be as in the previous section, then
$$
\sum_{\Nm_{k_r/k}(x)=\mu}\chi(\Nm_{k_r/k}(g(x)))=\sum_{P(x_1,\ldots,x_r)=\mu}\chi(\Nm_{k_r/k}(g(\alpha_1x_1+\cdots+\alpha_rx_r)))=
$$
$$
=\sum_{P(x_1,\ldots,x_r)=\mu}\chi\left(\prod_\sigma g^\sigma(\sigma(\alpha_1)x_1+\cdots+\sigma(\alpha_r)x_r)\right)
$$
so, by Grothendieck's trace formula,
\begin{equation}\label{lefschetz2}
\sum_{\Nm_{k_r/k}(x)=\mu}\chi(\Nm_{k_r/k}(g(x)))=\sum_{i=0}^{2r-2}\Tr(\Frob_k|\HH^i_c(V_\mu\otimes\bar k,\LL_{\chi(H)}))
\end{equation}
where $V_\mu$ is the same as in the previous section and $H(x_1,\ldots,x_r)=\prod_\sigma g^\sigma(\sigma(\alpha_1)x_1+\cdots+\sigma(\alpha_r)x_r)$, the product taken over the elements of $\Gal(k_r/k)$.

Over $k_r$, the map $(x_1,\ldots,x_r)\mapsto (\sigma(\alpha_1)x_1+\cdots+\sigma(\alpha_r)x_r)_{\sigma\in\Gal(k_r/k)}$ is an isomorphism betweem $\AAA^r_{k_r}$ and $\AAA^{\Gal(k_r/k)}_{k_r}$, and the pull-back of $P$ under this automorphism is $x_1\cdots x_r$. So $V_\mu\otimes\bar k$ is isomorphic to the hypersurface $x_1\cdots x_r=\mu$, and the sheaf $\LL_{\chi(H)}$ corresponds under this isomorphism to the sheaf $\LL_{\chi(\prod_\sigma g^\sigma(x_\sigma))}=\boxtimes_\sigma\LL_{\chi(g^\sigma)}$ where $\LL_{\chi(g^\sigma)}$ is the pull-back of the Kummer sheaf $\LL_\chi$ by $g^\sigma$. Thus $\dim\HH^i_c(V_\mu\otimes\bar k,\LL_{\chi(H)})=\dim\HH^i_c(\{x_1\cdots x_r=\mu\},\boxtimes_\sigma\LL_{\chi(g^\sigma)})$. By proper base change, the group $\HH^i_c(\{x_1\cdots x_r=\mu\},\boxtimes_\sigma\LL_{\chi(g^\sigma)})$ is the fibre at $\mu$ of the sheaf $\R^im_!(\boxtimes_\sigma\LL_{\chi(g^\sigma)})$, where $m:\AAA_{\bar k}^{\Gal(k_r/k)}\to\AAA^1_{\bar k}$ is the multiplication map. 

\begin{prop} Let $g_1,\ldots,g_r\in k_r[x]$ be square-free of degree $d$ with $g_i(0)\neq 0$, $m:\AAA_{k_r}^r\to\AAA^1_{k_r}$ the multiplication map and $K_r:=\R m_!(\LL_{\chi(g_1)}\boxtimes\cdots\boxtimes\LL_{\chi(g_r)})$. Suppose that $\chi^d$ is not trivial. Then $K_r=\LL_r[1-r]$ for a middle extension sheaf $\LL_r$ of generic rank $rd^{r-1}$ and pure of weight $r-1$ (on the open set where it is smooth), which is totally ramified at infinity and unipotent at $0$, with $\HH^1_c(\AAA^1_{\bar k},\LL_r)$ pure of weight $r$ and dimension $(d-1)^r$.
\end{prop}

\begin{proof}
 We will proceed by induction, as in \cite[Th\'eor\`eme 7.8]{deligne569application}. For $r=1$, $\LL_r=\LL_{\chi(g_1)}$ and all results are well known (see e.g. \cite{katz2002estimates}). The sheaf is smooth of rank $1$ on the complement of the set of roots of $g_1$, and the monodromy group at a root $\alpha$ acts via the non-trivial character $\chi$, so $\LL_{\chi(g_1)}$ is a middle extension at $\alpha$.  

Suppose everything has been proven for $r-1$. Then
$$
K_r=\R m_!(\LL_{\chi(g_1)}\boxtimes\cdots\boxtimes\LL_{\chi(g_r)})=\R m_{2!}(\R m_{1!}(\LL_{\chi(g_1)}\boxtimes\cdots\boxtimes\LL_{\chi(g_{r-1})})\boxtimes\LL_{\chi(g_r)})=
$$
$$
=\R m_{2!}(K_{r-1}\boxtimes\LL_{\chi(g_r)})=\R m_{2!}(\LL_{r-1}[2-r]\boxtimes\LL_{\chi(g_r)})
$$
where $m_1:\AAA^{r-1}_{k_r}\to\AAA^1_{k_r}$ and $m_2:\AAA^2_{k_r}\to\AAA^1_{k_r}$ are the multiplication maps.

The fibre of $K_r$ at $t\in\bar k$ is then $\R\Gamma_c(\{xy=t\}\subseteq\AAA^2_{\bar k},\LL_{r-1}\boxtimes\LL_{\chi(g_r)})[2-r]$. If $t\neq 0$, $\{xy=t\}$ is isomorphic to $\GG_m$ via the projection on $x$, so the fibre is $\R\Gamma_c(\GG_{m,\bar k},\LL_{r-1}\boxtimes\sigma_t^\star\LL_{\chi(g_r)})[2-r]$, where $\sigma_t:\GG_{m,\bar k}\to\GG_{m,\bar k}$ is the involution $x\mapsto t/x$. Since $\LL_{r-1}$ is totally ramified at $0$ (and unramified at infinity) and $\sigma_t^\star\LL_{\chi(g_r)}$ is unramified at $0$ (and totally ramified at infinity), their tensor product is totally ramified at both $0$ and infinity. In particular, its $\HH^2_c$ is vanishes. On the other hand, $\LL_{r-1}$ and $\LL_{\chi(g_r)}$ do not have punctual sections \cite[Corollary 6 and Proposition 9]{katz2003semicontinuity}, so neither does $\LL_{r-1}\otimes\sigma_t^\star\LL_{\chi(g_r)}$ and thus its $\HH^0_c$ vanishes. We conclude that the restriction of $K_r$ to $\GG_m$ is a single sheaf placed in degree $1+(r-2)=r-1$.

The fibre of $K_r$ at $0$ is $\R\Gamma_c(\{xy=0\}\subseteq\AAA^2_{\bar k},\LL_{r-1}\boxtimes\LL_{\chi(g_r)})[2-r]$. The group $\HH^2_c(\{xy=0\},\LL_{r-1}\boxtimes\LL_{\chi(g_r)})$ vanishes, because so does $\HH^2_c$ of its restriction to $x=0$ (which is a constant times $\LL_{\chi(g_r)}$, totally ramified at infinity) and to $y=0$ (which is a constant times $\LL_{r-1}$, also totally ramified at infinity). The group $\HH^0_c$ also vanishes, because neither the restiction of $\LL_{r-1}\boxtimes\LL_{\chi(g_r)}$ to $x=0$ nor its restriction to $y=0$ have punctual sections. So the stalk of $K_r$ at $0$ is also concentrated in degree $r-1$.

Once we know $K_r$ is a single sheaf $\LL_r=\R^{r-1}m_!(\LL_{\chi(g_1)}\boxtimes\cdots\boxtimes\LL_{\chi(g_r)})$, since $\HH^i_c(\AAA^1_{\bar k},\LL_{\chi(g_i)})=0$ for $i\neq 1$ and has dimension $d-1$ and is pure of weight $1$ for $i=1$ we get, by K\"unneth, that $\HH^i_c(\AAA^1_{\bar k},\LL_r)=0$ for $i\neq 1$ and it has dimension $(d-1)^r$ and is pure of weight $r$ for $i=1$. Similarly, since the inverse image of $\GG_{m,\bar k}$ under the multiplication map is $\GG_{m,\bar k}^r$, $\HH^1_c(\GG_{m,\bar k},\LL_r)=0$ for $i\neq 1$ and it has dimension $d^r$ for $i=1$. In particular, the rank of $\LL_r$ at $0$ is $\chi(\AAA^1_{\bar k},\LL_r)-\chi(\GG_{m,\bar k},\LL_r)=d^r-(d-1)^r$.

Let $t\in \bar k$ be a point which is not the product of a ramification point of $\LL_r$ and a ramification point of $\LL_{\chi(g_r)}$. Then at every point of $\GG_{m,\bar k}$ at least one of $\LL_{r-1}$, $\sigma_t^\star\LL_{\chi(g_r)}$ is smooth. Since $\LL_{r-1}$ has unipotent monodromy at $0$ and $\sigma_t^\star\LL_{\chi(g_r)}$ is unramified at $\infty$, by the Ogg-Shafarevic formula we have
$$
-\chi(\GG_{m,\bar k},\LL_{r-1})=\Swan_\infty\LL_{r-1}+\sum_{s\in \bar k^\star}(\Swan_s\LL_{r-1}+\mathrm{drop}_s\LL_{r-1})
$$
and
$$
-\chi(\GG_{m,\bar k},\sigma_t^\star\LL_{\chi(g_r)})=\Swan_0\LL_{\chi(g_r)}+\sum_{s\in \bar k^\star}(\Swan_{t/s}\LL_{\chi(g_r)}+\mathrm{drop}_{t/s}\LL_{\chi(g_r)})
$$
The local term at $u\in\bar k^\star$ (sum of the Swan conductor and the drop of the rank) gets multiplied by $e$ upon tensoring with un unramified sheaf of rank $e$. The local term at $0$ or $\infty$ (the Swan conductor) gets multiplied by $e$ upon tensoring with a sheaf of rank $e$ with unipotent monodromy. We conclude that
$$
-\chi(\GG_{m,\bar k},\LL_{r-1}\otimes\sigma_t^\star\LL_{\chi(g_r)})=-(\mathrm{rank}\;\LL_{\chi(g_r)})\chi(\GG_{m,\bar k},\LL_{r-1})-
$$
$$
-(\mathrm{rank}\;\LL_{r-1})\chi(\GG_{m,\bar k},\sigma_t^\star\LL_{\chi(g_r)})=d^{r-1}+d(r-1)d^{r-2}=rd^{r-1}.
$$
This is the generic rank of $\LL_r$.

Being a middle extension is a local property which is invariant under tensoring by unramified sheaves. Since, at every point of $\GG_{m,\bar k}$, at least one of $\LL_{r-1}$, $\sigma_t^\star\LL_{\chi(g_r)}$ is unramified and they are both middle extensions (by the induction hypothesis), their tensor product is a middle extension on $\GG_{m,\bar k}$. Since it is totally ramified at both $0$ and $\infty$, we conclude that $\HH^1_c(\GG_{m,\bar k},\LL_{r-1}\otimes\sigma_t^\star\LL_{\chi(g_r)})$ is pure of weight $(r-2)+1=r-1$ \cite[Th\'eor\` eme 3.2.3]{deligne1980conjecture}. So $\LL_r$ is pure of weight $r-1$ on the open set where it is smooth.

Now let $j_W:W\hookrightarrow\AAA^1_{\bar k}$ be the inclusion of the largest open sen on which $\LL_r$ is smooth. Since $\LL_r$ has no punctual sections, there is an injection $0\to\LL_r\to j_{W\star} j_W^\star\LL_r$, let $\mathcal Q$ be its punctual cokernel. We have an exact sequence
$$
0\to\HH^0_c(\AAA^1_{\bar k},{\mathcal Q})\to\HH^1_c(\AAA^1_{\bar k},\LL_r)\to\HH^1_c(\AAA^1_{\bar k},j_{W\star}j_W^\star\LL_r)\to 0
$$
where $\HH^0_c(\AAA^1_{\bar k},{\mathcal Q})$ has weight $\leq r-1$. Since $\HH^1_c(\AAA^1_{\bar k},\LL_r)$ is pure of weight $r$, we conclude that $\HH^0_c(\AAA^1_{\bar k},{\mathcal Q})$ and therefore $\mathcal Q$ are zero, so $\LL_r$ is a middle extension. Now let $j:\AAA^1_{\bar k}\hookrightarrow\PP^1_{\bar k}$ be the inclusion, again we get an exact sequence
$$
0\to\LL_r^{I_\infty}\to \HH^1_c(\AAA^1_{\bar k},\LL_r)\to \HH^1_c(\PP^1_{\bar k},j_\star\LL_r)\to 0
$$
with $\LL_r^{I_\infty}$ of weight $\leq r-1$, since $\HH^1_c(\AAA^1_{\bar k},\LL_r)$ is pure of weight $r$ we conclude that $\LL_r^{I_\infty}=0$, that is, $\LL_r$ is totally ramified at infinity.

It remains to prove that $\LL_r$ has unipotent monodromy at zero. Consider the exact sequence
$$
0\to\LL_r^{I_0}\to\HH^1_c(\GG_{m,\bar k},\LL_r)\to\HH^1_c(\AAA^1_{\bar k},\LL_r)\to 0
$$
which identifies $\LL_r^{I_0}$ with the weight $<r$ part of $\HH^1_c(\GG_{m,\bar k},\LL_r)$. Since $\HH^1_c(\GG_{m,\bar k},\LL_r)=\bigotimes_{i=1}^r\HH^1_c(\GG_{m,\bar k},\LL_{\chi(g_i)})$ and $\HH^1_c(\GG_{m,\bar k},\LL_{\chi(g_i)})$ has $d-1$ Frobenius eigenvalues of weight $1$ and one of weight $0$, we conclude that $\HH^1_c(\GG_{m,\bar k},\LL_r)$ has ${r\choose i}(d-1)^i$ eigenvalues of weight $i$ for every $i=0,\ldots,r$. By \cite[Theorem 7.0.7]{katz1988gsk}, an eigenvalue of weight $i<r$ on $\LL_r^{I_0}$ corresponds to a unipotent Jordan block of size $r-i$ for the action of $I_0$. So the sum of the sizes of the unipotent Jordan blocks for the monodromy of $\LL_r$ at $0$ is
$$
\sum_{i=0}^{r-1}{r\choose i}(d-1)^i(r-i)=r\sum_{i=0}^{r-1}{r\choose i}(d-1)^i-r\sum_{i=0}^{r-1}{{r-1}\choose{i-1}}(d-1)^i=
$$
$$
=r\sum_{i=0}^{r-1}{{r-1}\choose i}(d-1)^i=r(1+d-1)^{r-1}=rd^{r-1}
$$
which is the generic rank of $\LL_r$. So the unipotent Jordan blocks fill out the entire monodromy at $0$. 
\end{proof}

\begin{cor}
 Suppose that $g$ is square-free of degree $d$ prime to $p$ and $\chi^d$ is not trivial. For any $\mu\in k^\star$, $\HH^i_c(V_\mu\otimes\bar k,\LL_{\chi(H)})=0$ for $i\neq r-1$ and $\dim\HH^{r-1}_c(V_\mu\otimes\bar k,\LL_{\chi(H)})=rd^{r-1}$.
\end{cor}

\begin{proof}
 Apply the previous proposition with $(g_1,\ldots,g_r)=(g^\sigma)_{\sigma\in\Gal(k_r/k)}$, and proper base change. 
\end{proof}

\begin{cor}
 Suppose that $g$ is square-free of degree $d$ prime to $p$ and $\chi^d$ is not trivial. Then
$$
\left|\sum_{x\in k_r^\star}\chi(\Nm_{k_r/k}(f(x)))\right|\leq rd^{r-1}(q-1)q^{\frac{r-1}{2}}
$$
\end{cor}

\begin{proof}
 Since $\LL_{\chi(H)}$ is pure of weight $0$, $\HH^{r-1}_c(V_\mu\otimes\bar k,\LL_{\chi(H)})$ has weights $\leq r-1$ for every $\mu$. So the previous corollary together with (\ref{lefschetz2}) implies 
$$
\left|\sum_{\Nm_{k_r/k}(x)=\mu}\chi(\Nm_{k_r/k}(g(x)))\right|\leq rd^{r-1}q^{\frac{r-1}{2}}
$$
for every $\mu\in k^\star$. We conclude by using (\ref{suma2}).
\end{proof}

\begin{rem}\emph{
 The following example shows that the hypothesis $\chi^d$ non-trivial is necessary. Let $p$ be odd, $r=2$, $g(x)=x^2+1$ and $\rho:k^\star\to\QQ^\star$ the quadratic character. Then
$$
\sum_{\Nm_{k_r/k}(x)=1}\rho(\Nm_{k_r/k}(x^2+1))=\sum_{x^{q+1}=1}\rho((x^2+1)(x^{2q}+1))=
$$
$$
=\sum_{x^{q+1}=1}\rho(x^2+x^{2q}+2)=\sum_{x^{q+1}=1}\rho((x+x^q)^2)\geq q-1
$$
since $x+x^q=\Tr_{k_r/k}(x)\in k$ and therefore $\rho((x+x^q)^2)=\rho(x+x^q)^2=1$ unless $x+x^q=0$, which only happens for $x^2=-1$, that is, for at most two values of $x$. So we can never have an estimate of the form
$$
\left|\sum_{\Nm_{k_r/k}(x)=1}\rho(\Nm_{k_r/k}(x^2+1))\right|\leq C\cdot q^{\frac{1}{2}}
$$
which is valid for all $q$.}
\end{rem}

\bibliographystyle{amsplain}
\bibliography{bibliography}

\end{document}